\documentclass[a4paper, 10pt]{amsart}
\usepackage{amscd}
\usepackage{amssymb}
\usepackage[all]{xy}
\usepackage{mathpazo}
\date{date 12 February 2010}
\title[Flatness and Completion]{On Flatness and Completion for
Infinitely Generated Modules over Noetherian Rings}

\author{Amnon Yekutieli}
\address{Department of  Mathematics, Ben Gurion University,
Be'er Sheva 84105,
Israel}
\email{amyekut@math.bgu.ac.il}
\thanks{
Primary: 13J10; Secondary: 13E05, 13C10, 13C11.}
\thanks{This research was supported by the Israel Science Foundation.}

\newtheorem{thm}[equation]{Theorem}
\newtheorem{cor}[equation]{Corollary}
\newtheorem{prop}[equation]{Proposition}
\newtheorem{lem}[equation]{Lemma}
\theoremstyle{definition}
\newtheorem{dfn}[equation]{Definition}
\newtheorem{rem}[equation]{Remark}
\newtheorem{exa}[equation]{Example}

\numberwithin{equation}{section}

\newcommand{\xar}{\xrightarrow}
\newcommand{\opn}{\operatorname}
\newcommand{\cat}[1]{\operatorname{\mathsf{#1}}}

\newcommand{\rmitem}[1]{\item[\text{\textup{(#1)}}]}
\newcommand{\mfrak}[1]{\mathfrak{#1}}
\newcommand{\mcal}[1]{\mathcal{#1}}

\newcommand{\mbf}[1]{\mathbf{#1}}
\newcommand{\mrm}[1]{\mathrm{#1}}
\newcommand{\mbb}[1]{\mathbb{#1}}

\newcommand{\smfrac}[2]{\textstyle \frac{#1}{#2}}
\newcommand{\tup}[1]{\textup{#1}}

\newcommand{\boplus}{\bigoplus\nolimits}

\newcommand{\bosum}{\sum\nolimits}
\newcommand{\boprod}{\prod\nolimits}

\newcommand{\til}[1]{\tilde{#1}}
\newcommand{\what}[1]{\widehat{#1}}
\newcommand{\hatotimes}[1]{\, \what{\otimes}_{#1} \,}
\newcommand{\bra}[1]{\langle #1 \rangle}

\newcommand{\K}{\mbb{K}}
\newcommand{\R}{\mbb{R}}
\newcommand{\N}{\mbb{N}}

\newcommand{\m}{\mfrak{m}}
\renewcommand{\a}{\mfrak{a}}
\newcommand{\n}{\mfrak{n}}

\begin{document}

\begin{abstract}
Let $A$ be a noetherian commutative ring, and let $\a$ be an ideal in $A$. We
study questions of flatness and $\a$-adic completeness for
infinitely generated $A$-modules. This is done using the notions of
{\em decaying function} and {\em $\a$-adically free $A$-module}. 
\end{abstract}

\maketitle
\tableofcontents

\setcounter{section}{-1}
\section{Introduction}

Let $A$ be a commutative ring, and let $\a$ be an ideal of $A$.
For $i \geq 0$ we write $A_i := A / \a^{i+1}$ . Given an $A$-module $M$, its 
$\a$-adic completion is the $A$-module
\begin{equation} \label{eqn:17}
\what{M} := \lim_{\leftarrow i}\, (A_i \otimes_A M) .
\end{equation}
Recall that $M$ is called {\em $\a$-adically complete}
if the canonical homomorphism $M \to \what{M}$ is bijective. 
If $M \to \what{M}$ is injective then $M$ is called 
{\em $\a$-adically separated}.
It is well known that if $A$ is noetherian and complete, then all finitely
generated $A$-modules are complete. But for infinitely generated modules, and
for non-noetherian rings, the picture is quite complicated.

We became interested in the adic completion of infinitely generated modules
in the course of our work on deformation quantization (see end of
introduction). After a while we realized that this old and apparently
simple concept was not treated adequately in the literature. This paper
contains our contributions.

In Section \ref{sec:comp} we discuss the completion operation in general. 
In Theorem \ref{thm:6} we give a useful criterion to tell whether the
$\a$-adic completion $\what{M}$ of an $A$-module $M$ is itself $\a$-adically
complete. We give an example of an $\a$-adically separated module $M$ whose 
$\a$-adic completion $\what{M}$ is not complete (Example \ref{exa:7}). 
The moral (made precise in Corollary \ref{cor:18})
is that one should distinguish between the algebraic notion of
$\a$-adic completion of $M$ (i.e.\ the inverse limit (\ref{eqn:17})),
and the topological notion of completion of the metric space $M$ (with respect
to its $\a$-adic metric, see (\ref{eqn:16})).

In Section \ref{sec:dec} of the paper we introduce the notion of 
{\em decaying function}. This idea is inspired by functional analysis.
Let $Z$ be a set, and let $M$ be an $\a$-adically separated $A$-module. A
function $f : Z \to M$ is called decaying if for every 
$i$ the composed function $Z \to A_i \otimes_A M$ has finite support. 
We denote by $\opn{F}_{\mrm{dec}}(Z, M)$ the set of all decaying
functions  $f : Z \to M$, and this is an $A$-module in the obvious way. The
submodule of finite support functions is denoted by 
$\opn{F}_{\mrm{fin}}(Z, M)$.
Note that for $M := A$ the module $\opn{F}_{\mrm{fin}}(Z, A)$
is a free $A$-module, with basis the collection 
$\{ \delta_z \}_{z \in Z}$ of delta functions.

We prove (Corollary \ref{cor:7}) 
that if $M$ is $\a$-adically complete, then
$\opn{F}_{\mrm{dec}}(Z, M)$ 
is the $\a$-adic completion of
$\opn{F}_{\mrm{fin}}(Z, M)$. (Recall however that the completion need not be
complete!) We also prove a complete version of the Nakayama Lemma (Theorem
\ref{thm:8}).

In Section \ref{sec:noeth} we assume $A$ is noetherian.
The main result here, Theorem \ref{thm:2}, says that for any set $Z$
the $A$-module $\opn{F}_{\mrm{dec}}(Z, \what{A})$ is flat and $\a$-adically
complete. Theorem \ref{thm:2} implies, among other things, 
Corollary \ref{cor:3}, which says that for any
$A$-module $M$ the completion $\what{M}$ is $\a$-adically complete.
(Note that the content of Corollary \ref{cor:3} is not new; see Remark
\ref{rem:2} for a bit of history.)
We see that the anomalies of completion disappear when $A$ is noetherian.

An $A$-module $P$ is called {\em $\a$-adically free} if it is isomorphic to
$\opn{F}_{\mrm{dec}}(Z, \what{A})$ for some set $Z$.  
We show (Corollary \ref{cor:6})
that any $\a$-adically complete $A$-module $M$ is a quotient of some 
$\a$-adically free module $P$. We also introduce the notion of
{\em $\a$-adically projective} $A$-module; and we prove that $P$ is
$\a$-adically projective if and only if it is a direct summand of an
$\a$-adically free module (Corollary \ref{cor:15}).
We give an example (Example \ref{exa:5}) demonstrating that the completion
functor $M \mapsto \what{M}$ is not right exact.

In Section \ref{sec:local} we specialize to the case of a complete noetherian
local ring $A$, with maximal ideal $\m$. Corollary \ref{cor:1} says that an
$A$-module $P$ is $\m$-adically free if and only if it is flat and $\m$-adically
complete. We discuss $\m$-adic systems of $A$-modules.

In Section \ref{sec:sheaves} we study the related geometric problem.
Namely $X$ is a topological space, and we are interested in sheaves
of $A$-modules on $X$ that are flat and $\m$-adically complete. 
Here some geometric property is needed for things to work well; we call it 
{\em locally $\mcal{N}$-simply connectedness}, where
$\mcal{N}$ is a sheaf of abelian groups on $X$ (see Definition \ref{dfn:20}).

Here are a few words on the connection between completion and deformation
quantization. Suppose $\K$ is a field, and $A$ is a complete noetherian local
$\K$-algebra, with maximal ideal $\m$, such that $A / \m \cong \K$. 
Let $\bar{B}$ be a $\K$-algebra. An {\em associative $A$-deformation} of
$\bar{B}$ is an associative unital (but not necessarily commutative)
$A$-algebra $B$, which is flat and $\m$-adically complete, together with a
$\K$-algebra isomorphism $\K \otimes_A B \cong \bar{B}$.
The main example is $\K := \R$; $A := \R[[\hbar]]$, the ring of formal power
series in the variable $\hbar$; and $\bar{B} := \mrm{C}^{\infty}(X)$, the ring 
of smooth functions on a differentiable manifold $X$. 
In our paper \cite{Ye2} we
consider the algebro-geometric version of deformation quantization, involving
sheaves of $A$-algebras. The results of Sections 4-5 are needed in \cite{Ye2}.

A possible use for the results of Section 3 would be to gain a better
understanding of the Matlis-Greenlees-May duality (cf.\ \cite{Ml}, \cite{GM} and
\cite{AJL}).

\medskip \noindent
{\bf Acknowledgments.} 
I wish to thank V. Berkovich, V. Drinfeld, M. Hochster, L. Avramov, J.R.
Strooker, A.M. Simon, Liran Shaul and Marco Porta for discussions on this
material.

\section{Some Results about Completion}
\label{sec:comp}

By default all rings in this paper are commutative. 

We begin by recalling some facts about completion.
Let $A$ be a ring, and let $\a$ be an ideal of $A$.
For $i \in \mbb{N}$ we write $A_i := A / \a^{i+1}$.
Given an $A$-module $M$, there are canonical isomorphisms 
$A_i \otimes_A M \cong M / \a^{i+1} M$.
The {\em $\a$-adic completion} of $M$ is the $A$-module
\begin{equation} \label{eqn:11}
\what{M} := \lim_{\leftarrow i}\, (A_i \otimes_A M)  .
\end{equation}
There is a canonical homomorphism
\[ \tau_M : M \to \what{M} . \]
The module $M$ is called {\em $\a$-adically separated} if $\tau_M$ is
injective, and it is called {\em $\a$-adically complete} if 
$\tau_M$ is bijective.
(Some texts, such as \cite{CA}, would say that $A$ is separated and complete
if $\tau_M$ is bijective.)
Of course $M$ is $\a$-adically separated if and only if 
$\bigcap_{i \geq 0} \a^i M = 0$.
If $M$ is $\a$-adically complete, then we often identify $M$ with $\what{M}$
via $\tau_M$. 

The $\a$-adic completion $\what{A}$ of $A$ is a ring, and 
$\tau_A : A \to \what{A}$ is ring homomorphism.
Given an $A$-module $M$, its completion $\what{M}$ is an $\what{A}$-module, with
action coming from the action of $\what{A}$ on the modules 
$A_i \otimes_A M$ in the inverse system (\ref{eqn:11}).
In particular this says that a complete $A$-module $M$ has a canonical 
$\what{A}$-module structure on it.

Given a homomorphism $\phi : M \to N$ of $A$-modules, there is an induced
homomorphism 
$\what{\phi} : \what{M} \to \what{N}$
making the diagram 
\[ \UseTips \xymatrix @C=5ex @R=5ex {
M
\ar[r]^{\phi}
\ar[d]_{\tau_M}
& 
N
\ar[d]^{\tau_N}
\\
\what{M}
\ar[r]^{\what{\phi}}
&
\what{N}
} \]
commutative. 

Sometimes we write
$\Lambda_{\a} M := \what{M}$
for an $A$-module $M$, and 
$\Lambda_{\a}(\phi) := \what{\phi}$
for a homomorphism $\phi$, following \cite{AJL}. 
This gives a functor
\[ \Lambda_{\a} : \cat{Mod} A \to \cat{Mod} A  \]
on the category of $A$-modules. The functor $\Lambda_{\a}$ is additive.
However it is not exact, nor even right exact; cf.\ Example 
\ref{exa:5}. The functor $ \Lambda_{\a}$ is not idempotent in general
(see Example \ref{exa:7}).  Corollary \ref{cor:16} says that the functor
$\Lambda_{\a}$ is idempotent if the ideal $\a$ is finitely generated. All that
can be said in general about the functor $\Lambda_{\a}$ is that it preserves
surjections:

\begin{prop} \label{prop:10}
Let $\phi : M \to N$ be a surjective homomorphism of $A$-modules. Then
$\what{\phi} : \what{M} \to \what{N}$ is also surjective.
\end{prop}

This result is part of \cite[Proposition 2.2.1]{St}.

\begin{proof}
For every $i \geq 0$ let us write 
$M_i := A_i \otimes_A M$ and
$N_i := A_i \otimes_A N$.
Let $\phi_i : M_i \to N_i$ be the homomorphism induced by $\phi$, and let
$K_i := \opn{Ker}(\phi_i)$. 
So there is an inverse system of exact sequences
\[ 0 \to K_i \to M_i \xar{\phi_i} N_i \to 0 . \]
Each $K_i$ is a quotient of $\opn{Ker}(\phi)$, and therefore 
$K_{i+1} \to K_i$ is surjective. By the Mittag-Leffler argument
(as in \cite[Proposition 10.2]{AM}),
in the limit we get an exact sequence
\[ 0 \to \lim_{\leftarrow i} K_i \to \what{M} \xar{\what{\phi}} 
\what{N} \to  0 . \]
In particular $\what{\phi}$ is surjective.
\end{proof}

Let $M$ be an $A$-module. The homomorphism $\tau_M : M \to \what{M}$ induces a
homomorphism
\begin{equation} \label{eqn:19}
\tau_{M, i} : A_i \otimes_A M \to A_i \otimes_A \what{M}
\end{equation}
for every $i \geq 0$. On the other hand from the inverse limit (\ref{eqn:11})
we have surjective homomorphisms
\begin{equation} \label{eqn:20}
\pi_{M, i} : \what{M} \to A_i \otimes_A M . 
\end{equation}

Here is a useful criterion to tell whether the $\a$-adic completion is complete.

\begin{thm} \label{thm:6}
Let $M$ be an $A$-module, with $\a$-adic completion $\what{M}$. 
The following conditions are equivalent\tup{:}
\begin{itemize}
\rmitem{i} The $A$-module $\what{M}$ is $\a$-adically complete.
\rmitem{ii} All the homomorphisms $\tau_{M, i}$ are surjective.
\rmitem{iii} There is equality $\opn{Ker}(\pi_{M, i}) = \a^{i+1} \what{M}$
for every $i \geq 0$.
\end{itemize}
\end{thm}

\begin{proof}
The proof is based on ideas in \cite[Section 2.2]{St}.  Let us write 
$N := \what{M}$, and for $i \geq 0$ let
$M_i := A_i \otimes_A M$ and $N_i := A_i \otimes_A N$.
There is a commutative diagram
\begin{equation} \label{eqn:14}
\UseTips \xymatrix @C=7ex @R=7ex {
M
\ar[r]^{\tau_M}
\ar[d]^{\theta_{M, i}}
& 
N
\ar[d]^{\theta_{N, i}}
\ar[r]^{\pi_{M, i}}
& 
M_i
\ar[d]^{=}
\\
M_i
\ar[r]^{\tau_{M, i}}
&
N_i
\ar[r]^{\psi_i}
&
M_i
}
\end{equation}
in which $\theta_{M, i}$ and $\theta_{N, i}$ are the surjections induced by the
ring homomorphism $A \to A_i$, and $\psi_i$ is the unique
homomorphism that makes the  diagram commutative.
Since $\psi_i \circ \tau_{M, i}$ is the identity on 
$M_i$, it follows that $\tau_{M, i}$ is a split injection.
Letting $M'_i := \opn{Ker}(\psi_i)$, we have a canonical decomposition
$N_i = M_i \oplus M'_i$. So $\tau_{M, i}$ is surjective if and only if
 $\psi_i$ is injective, if and only if $M'_i = 0$.

Note also that 
$\opn{Ker}(\theta_{N, i}) = \a^{i+1} N$, and it equals
$\opn{Ker}(\pi_{M, i})$ if and only if $\psi_i$ is injective. This tells us
that (ii) $\Leftrightarrow$ (iii).

The diagrams (\ref{eqn:14}) form an inverse system. Passing to the inverse
limit 
in the second row, we get a diagram
\[ \UseTips \xymatrix @C=7ex @R=5ex {
N
\ar[r]^{\tau_{N}}
&
\what{N}
\ar[r]^{\psi}
&
N
} \]
where $\psi := \lim_{\leftarrow i}\, \psi_i$.
Again $\psi \circ \tau_N = \opn{id}_N$, so $\tau_N$ is a split injection.
Writing 
$N' := \opn{Ker}(\psi)$, we have a canonical decomposition
$\what{N} = N \oplus N'$. So $\tau_N$ is bijective (i.e.\ $N$ is complete)
if and only $N' = 0$.

The decompositions 
$N_i = M_i \oplus M'_i$
are compatible as $i$ varies, and hence
\[ N' \cong  \lim_{\leftarrow i}\, M'_i . \]
Now in the inverse system $\{ M'_i \}_{i \geq 0}$ 
the homomorphisms $M'_{j} \to M'_i$, for $j \geq i$, are surjective. 
Therefore in the limit, every homomorphism $N' \to M'_i$ is surjective.
It follows that $N' = 0$ if and only if $M'_i = 0$ for all $i$.
We conclude that (i) $\Leftrightarrow$ (ii).
\end{proof}

In the proof above we also showed that:

\begin{cor} \label{cor:17}
Let $M$ be an $A$-module. Then its $\a$-adic completion $\what{M}$ is 
$\a$-adically separated. Moreover, the homomorphism
$\tau_{\what{M}} : \what{M} \to \Lambda_{\a} \what{M}$ is a split injection.
\end{cor}

Some examples of the bad behavior of completion can be found in the literature.
Strooker \cite[Subsection 2.2.5]{St} mentions unpublished work of Bartijn. And
there is an example in \cite{CA}, which is very close to the example we now
present. 

\begin{exa} \label{exa:7}
Let $\K$ be a field, and let
$A := \K[ t_1, t_2, \ldots ]$, the ring of polynomials in
countably many variables. In it consider the maximal ideal
$\a = (t_1, t_2, \ldots )$.
We will produce {\em an $A$-module $M$ whose $\a$-adic completion
$\what{M}$ is not  $\a$-adically complete}. 
In fact we will take $M := A$, the free module of rank $1$.

Let $\what{A}$ be the $\a$-adic completion of $A$, and let
$\mfrak{b} := \opn{Ker}(\pi_{A, 0} : \what{A} \to A_0)$.
The ring $\what{A}$ is canonically isomorphic to the ring of formal power series
$\K[[ t_1, t_2, \ldots ]]$.
In \cite[Exercise III.2.12]{CA} it is shown that  the ring 
$\what{A}$ is not $\mfrak{b}$-adically complete (when $\K$ is finite).
As stated in the previous paragraph, we will show something slightly different:
the $A$-module $\what{A}$ is not $\a$-adically complete (with no assumption on
the field $\K$). This is done using Theorem \ref{thm:6}. 

In order to utilize the notation of Theorem \ref{thm:6} and its proof,
let's write $M := A$ and $N := \what{M}$. To prove that $N$ is not $\a$-adically
complete it suffices to show that the homomorphism
$\tau_{M, 0} : M_0 \to N_0$
is not surjective.

Consider an element $b \in N$, with image 
$\bar{b} := \theta_{N,0}(b) \in N_0$.
The element $\bar{b}$ is in the image of $\tau_{M, 0}$ if and only if
$b \in \a N + \opn{Im}(\tau_M)$.
Now any element of  $\a N$ is of the form
$\sum_{k = 1}^n \, t_{k} b_k$ for some $n \geq 0$ and $b_k \in N$. 
And any element of $\opn{Im}(\tau_M)$ is a polynomial; so it lies in
$\K \oplus \a N$. Thus $b \in \a N + \opn{Im}(\tau_M)$ if and only if
\begin{equation} \label{eqn:13}
b =  \lambda + \sum_{k = 1}^n \, t_{k} b_k 
\end{equation}
for some $\lambda \in \K$, $n \geq 0$ and $b_k \in N$.

Let us take $b \in N = \K[[ t_1, t_2, \ldots ]]$ to be the power series
$b :=  \sum_{k = 1}^{\infty} t_{k}^{k}$.
Then $b$ cannot be written as in (\ref{eqn:13}), and hence
$\bar{b} \notin \opn{Im}(\tau_{M, 0})$.
\end{exa}

We end this section with a discussion of the topological interpretation of
$\a$-adic completion.
Any $A$-module $M$ has on it the $\a$-adic topology,
in which the collection of submodules $\{ \a^i M \}_{i \geq 0}$ is a basis of
open neighborhoods of the element $0$. Any homomorphism
$\phi : M \to N$ of $A$-modules is continuous for the $\a$-adic topologies.

Now consider an $\a$-adically separated $A$-module $M$. 
Recall that for an element $m \in M$ its {\em order}
(with respect to $\a$) is
\begin{equation} \label{eqn:18}
\opn{ord}_{\a}(m) := \opn{sup} \, \{ i \in \mbb{N} \mid
m \in \a^i M \} \in \mbb{N} \cup \{ \infty \} .
\end{equation}
Sometimes we shall write $\opn{ord}_{\a, M}(m)$, when we need to emphasize
the module $M$ (e.g.\ in the proof of Lemma \ref{lem:1}). 
Since $\bigcap_{i \geq 0} \a^i M = 0$, we see that 
$\opn{ord}_{\a}(m) = \infty$ if and only if $m = 0$. 
And $\opn{ord}_{\a}(m) = i \in \mbb{N}$ if and only if
$m \in \a^i M - \a^{i+1} M$.

For elements $m, n \in M$ define 
\begin{equation} \label{eqn:16}
\opn{dist}_{\a}(m, n) := (\smfrac{1}{2})^{\opn{ord}_{\a}(m-n)} \, .
\end{equation}
The function $\opn{dist}_{\a}$ is a metric on $M$, which we call 
the {\em $\a$-adic metric}. This metric determines the $\a$-adic topology
on $M$. The module $M$ is $\a$-adically complete
if and only if it is a complete metric space
with respect to the $\a$-adic metric.
See \cite[Section 10]{AM} or \cite[Section III.2.5]{CA}.

We continue with the assumption that $M$ is $\a$-adically separated; and we
view $M$ as a submodule of $\what{M}$ via the homomorphism $\tau_M$.
The $\a$-adically separated $A$-module
$\what{M}$ has on it two descending filtrations, defining two possibly distinct
metrics:
\begin{enumerate}
\rmitem{a} The filtration 
$\{ F^i \what{M} \}_{i \geq 0}$, where 
$F^i \what{M} := \opn{Ker}(\pi_{M, i-1})$
for $i \geq 1$, and $F^0 \what{M} := \what{M}$.
Here $\pi_{M, i-1}$ is the homomorphism in (\ref{eqn:20}).
There is a corresponding order function
\[ \opn{ord}'(m) := \opn{sup} \, \{ i \in \mbb{N} \mid
m \in F^i \what{M} \}  \]
for $m \in \what{M}$, and the corresponding metric is
\[ \opn{dist}'(m, n) := (\smfrac{1}{2})^{\opn{ord}'(m-n)} \]
for $m, n \in \what{M}$.

\rmitem{b} The filtration 
$\{ \a^{i} \what{M} \}_{i \geq 0}$, namely the $\a$-adic filtration of the
$A$-module $\what{M}$ itself. The corresponding order function 
$\opn{ord}_{\a, \what{M}}$ and metric $\opn{dist}_{\a, \what{M}}$
are given by formulas (\ref{eqn:18}) and (\ref{eqn:16}), replacing $M$ with 
$\what{M}$.
\end{enumerate}

The standard fact (see \cite[Section 10]{AM}) is that the metric space
$(\what{M}, \opn{dist}')$ is always the completion of the metric space
$(M, \opn{dist}_{\a})$. However:

\begin{cor} \label{cor:18}
Let $M$ be an $A$-module, with $\a$-adic completion $\what{M}$. The following
conditions are equivalent: 
\begin{itemize}
\rmitem{i} The $A$-module $\what{M}$ is $\a$-adically complete.
\rmitem{ii} The metrics $\opn{dist}_{\a, \what{M}}$ and
$\opn{dist}'$ on $\what{M}$ coincide.
\end{itemize}
\end{cor}

\begin{proof}
This is immediate from the equivalence
(i) $\Leftrightarrow$ (iii) in Theorem \ref{thm:6}.
\end{proof}

\begin{exa}
Consider the module $M$ from Example \ref{exa:7}. Since its $\a$-adic
completion $\what{M}$ is not $\a$-adically complete, we know that the metrics 
$\opn{dist}_{\a, \what{M}}$ and $\opn{dist}'$ are not the same. Indeed, a
little calculation shows that for the element
$b =  \sum_{k = 1}^{\infty} t_{k}^{k}$ we have
$\opn{dist}_{\a, \what{M}}(b, 0) = 1$, but
$\opn{dist}'(b, 0) = \smfrac{1}{2}$.
\end{exa}

\section{Modules of Decaying Functions}
\label{sec:dec}

The ideas in this section are inspired by functional analysis.
As in Section \ref{sec:comp}, $A$ is a ring and $\a$ is an ideal in it.

Let $Z$ be a set, and let $M$ be an $A$-module. We denote by 
$\opn{F}(Z, M)$ the set of all functions $f : Z \to M$, and by
$\opn{F}_{\mrm{fin}}(Z, M)$ the set of functions with finite support. 
So 
\[ \opn{F}(Z, M) \cong \boprod_{z \in Z} M  \]
and
\[ \opn{F}_{\mrm{fin}}(Z, M) \cong \boplus_{z \in Z} M . \]
The set $\opn{F}(Z, M)$ is an $A$-module, and 
$\opn{F}_{\mrm{fin}}(Z, M)$ is a submodule. 

Now let us look at the special case $M = A$. 
For every $z \in Z$ there is the delta function 
$\delta_z \in \opn{F}_{\mrm{fin}}(Z, A)$, namely
$\delta_z(z) := 1$, and $\delta_z(z') := 0$ for $z' \neq z$. 
The $A$-module $\opn{F}_{\mrm{fin}}(Z, A)$ is free; as basis we
can take the collection of elements $\{ \delta_z \}_{z \in Z}$.

Suppose $M$ is an $\a$-adically separated $A$-module, and $m \in M$.
Recall the $\a$-adic order $\opn{ord}_{\a}(m)$ from formula (\ref{eqn:18}).

\begin{dfn} \label{dfn:1}
Let $Z$ be a set and $M$ an $\a$-adically separated $A$-module.
A function $f : Z \to M$ is called {\em decaying}
if for every $i \in \mbb{N}$ the set
\[ \{ z \in Z \mid \opn{ord}_{\a} ( f(z) ) \leq i \} \]
is finite. We denote by $\opn{F}_{\mrm{dec}}(Z, M)$
the set of all decaying functions $f : Z \to M$.
\end{dfn}

The support of a decaying function is of course countable. 
Any function with finite support is decaying. Thus we have inclusions
\[ \opn{F}_{\mrm{fin}}(Z, M) \subset \opn{F}_{\mrm{dec}}(Z, M)
\subset \opn{F}_{}(Z, M) . \]
It is easy to see that $\opn{F}_{\mrm{dec}}(Z, M)$ is an
$A$-submodule of $\opn{F}(Z, M)$.

\begin{exa} \label{exa:1}
Suppose that $\a^i M = 0$ for some $i$. Then a decaying function has
finite support, and we have 
\[ \opn{F}_{\mrm{fin}}(Z, M) = \opn{F}_{\mrm{dec}}(Z, M) . \]
\end{exa}

\begin{exa} \label{exa:2}
Suppose $A$ is complete.
Take variables $t_1, \ldots, t_n$, and consider the ring of
restricted formal power series 
$A\{ t_1, \ldots, t_n \}$ as in \cite[Section III.4.2]{CA}.
Then as $A$-modules we have
$A\{ t_1, \ldots, t_n \} \cong 
\opn{F}_{\mrm{dec}}(\mbb{N}^n, A)$.
\end{exa}

\begin{prop}
Let $M$ be an $\a$-adically separated module. Then 
$\opn{F}_{\mrm{dec}}(Z, M)$ is $\a$-adically separated.
\end{prop}

\begin{proof}
Let $f : Z \to M$ be a decaying function and let
$a \in \a^i$. Then $a f(z) \in \a^i M$ for every $z \in Z$. 
We see that
\[ \a^i \cdot \opn{F}_{\mrm{dec}}(Z, M) \subset
\opn{F}_{\mrm{dec}}(Z, \a^i M) . \]
But $M$ is separated, so 
\[ \bigcap\nolimits_{i \geq 0} \, \a^i \cdot \opn{F}_{\mrm{dec}}(Z, M) \subset
\bigcap\nolimits_{i \geq 0} \, \opn{F}_{\mrm{dec}}(Z, \a^i M) = 
\opn{F}_{\mrm{dec}} \bigl( Z,\, \bigcap\nolimits_{i \geq 0} \,\a^i M \bigr) 
= 0 . \] 
\end{proof}

Let $\phi : M \to N$ be a homomorphism between $\a$-adically
separated $A$-modules. For any $m \in M$ we have
$\opn{ord}_{\a, N}(\phi(m)) \geq \opn{ord}_{\a, M}(m)$. Hence 
there is an induced $A$-linear homomorphism
\[ \opn{F}_{\mrm{dec}}(Z, M) \to 
\opn{F}_{\mrm{dec}}(Z, N) , \
f \mapsto \phi \circ f . \]
Let us denote by
$\cat{Mod}_{\mrm{sep}} A$ the full subcategory of $\cat{Mod} A$ consisting of
$\a$-adically separated modules; this is an additive category. We see that
for a fixed set $Z$ there is an additive functor
\[ \opn{F}_{\mrm{dec}}(Z, -) : \cat{Mod}_{\mrm{sep}} A \to 
\cat{Mod}_{\mrm{sep}} A . \]

Suppose $M$ is an $\a$-adically separated $A$-module.
Let $Z$ be a set, and let 
$f : Z \to M$ be a function. One says that the {\em series 
$\sum_{z \in Z} f(z)$ converges} in the $\a$-adic topology, to some element 
$m \in M$, if for any
natural number $i$ there is a finite subset
$Z_i \subset Z$, such that
\[ m - \sum_{z \in Z_i} f(z) \in \a^{i+1} M , \]
and
$f(z) \in \a^{i+1} M$ for all $z \notin Z_i$.
In this case one writes
\[ m = \sum_{z \in Z} f(z) . \]
Of course if the series converges then the sum $m$ is
unique. Cf.\ \cite[Section III.2.6]{CA}.

\begin{prop} \label{prop:1}
Let $M$ be an $\a$-adically complete $A$-module, and let
$f : Z \to M$ be a function. The following conditions are
equivalent\tup{:}
\begin{enumerate}
\rmitem{i} The function $f$ is decaying.
\rmitem{ii} The series $\sum_{z \in Z} f(z)$ converges in $M$ for
the $\a$-adic topology. 
\end{enumerate}
\end{prop}

The proof is easy, and we leave it out. An immediate consequence is that for an
$\a$-adically complete module $M$ there is an $A$-linear homomorphism
\[ \opn{F}_{\mrm{dec}}(Z, M) \to M , \ 
f \mapsto \sum_{z \in Z} f(z) . \]

\begin{cor} \label{cor:12}
Let $Z$ be a set, let $M$ be an $A$-module, and let $f : Z \to M$ be any
function.
Assume $A$ and $M$ are $\a$-adically complete. Then for any 
$g \in \opn{F}_{\mrm{dec}}(Z, A)$ the series 
$\sum_{z \in Z} g(z) f(z)$
converges in $M$. This gives rise to an $A$-linear homomorphism
\[ \phi : \opn{F}_{\mrm{dec}}(Z, A) \to M, \quad
\phi(g) := \sum_{z \in Z} g(z) f(z) . \]
\end{cor}

\begin{proof}
Given $g \in \opn{F}_{\mrm{dec}}(Z, A)$ consider the 
the function $Z \to M$, 
$z \mapsto g(z) f(z)$. This function is decaying, so by Proposition
\ref{prop:1} the series 
$\sum_{z \in Z} g(z) f(z)$ converges. 
It is easy to check that the resulting function $\phi$ is $A$-linear.
\end{proof}

\begin{thm} \label{thm:7}
Let $M$ be an $A$-module whose $\a$-adic completion
$\what{M}$ is $\a$-adically complete. Then the canonical homomorphism
\begin{equation} \label{equ:1}
\opn{F}_{\mrm{dec}}(Z, \what{M}) \to 
\lim_{\leftarrow i}\, \opn{F}_{\mrm{fin}}(Z, A_i \otimes_A M) 
\end{equation}
induced by 
\[ \pi_{M, i} : \what{M} \to A_i \otimes_A M \]
is bijective.
\end{thm}

\begin{proof}
Suppose $f \in \opn{F}_{\mrm{dec}}(Z, \what{M})$ is nonzero. So 
$f(z) \neq 0$ for some $z \in Z$. Since $\what{M}$ is separated, there is
some $i$ such that the image $\pi_{M, i}(f(z))$ of $f(z)$ in 
$A_i \otimes_A M$ is nonzero. This implies that the homomorphism
(\ref{equ:1}) is injective.

Conversely, suppose $\{ f_i \}_{i \geq 0}$ is an inverse system of
functions $f_i : Z \to A_i \otimes_A M$, each with finite
support. For any $z \in Z$ let 
\[ f(z) := \lim_{\leftarrow i}\, f_i(z) \in 
\lim_{\leftarrow i}\, (A_i \otimes_A M) = \what{M} . \] 
We get a function $f : Z \to \what{M}$ satisfying
$\pi_{M, i} \circ f = f_i$. Since each $f_i$ has finite support, and 
by Theorem \ref{thm:6} we know that 
$\opn{Ker}(\pi_{M, i}) = \a^{i+1} \what{M}$,
it follows that $f$ is a decaying function.
\end{proof}

\begin{cor} \label{cor:7}
Let $M$ be as in Theorem \tup{\ref{thm:7}}. Then the homomorphism
\[ \opn{F}_{\mrm{fin}}(Z, M) \to \opn{F}_{\mrm{dec}}(Z, \what{M}) \]
induced by $\tau_M$ makes $\opn{F}_{\mrm{dec}}(Z, \what{M})$ into an $\a$-adic
completion of $\opn{F}_{\mrm{fin}}(Z, M)$.
More precisely, there is an isomorphism
\[ \opn{F}_{\mrm{dec}}(Z, \what{M}) \cong 
\Lambda_{\a} \opn{F}_{\mrm{fin}}(Z, M) \]
that commutes with the homomorphisms from $\opn{F}_{\mrm{fin}}(Z, M)$, and is
functorial in $M$.
\end{cor}

The reason for the careful wording of the corollary is because 
$\opn{F}_{\mrm{dec}}(Z, \what{M})$ might fail to be $\a$-adically complete.
Cf.\ Example \ref{exa:7} and Corollary \ref{cor:18}.

\begin{proof}
Since there is a canonical isomorphism
\[ A_i \otimes_A \opn{F}_{\mrm{fin}}(Z, M) \cong 
\opn{F}_{\mrm{fin}}(Z, A_i \otimes_A M) \]
for every $i$, this follows from Theorem \ref{thm:7}.
\end{proof}

\begin{dfn}
Let $M$ be an $A$-module, and let 
$\{ m_z \}_{z \in Z}$ be a collection of elements of $M$.
Assume $A$ and $M$ are  $\a$-adically complete.  We say the
collection $\{ m_z \}_{z \in Z}$ {\em $\a$-adically generates} $M$ if for every
element $m \in M$ there exists some decaying function
$g : Z \to A$ such that 
\[ m = \sum_{z \in Z} g(z) m_z . \]
\end{dfn}

Here is a version of the Nakayama Lemma.

\begin{thm} \label{thm:8}
Let $M$ be an $A$-module, and let 
$\{ m_z \}_{z \in Z}$ be a collection of elements of $M$.
Assume $A$ and $M$ are  $\a$-adically complete.  
Write $M_0 := A_0 \otimes_A M$, and let
$\pi_0 : M \to M_0$ be the canonical homomorphism. 
Then the two conditions below are equivalent.
\begin{enumerate}
\rmitem{i} The collection $\{ \pi_0(m_z) \}_{z \in Z}$ generates the
$A_0$-module
$M_0$.
\rmitem{ii} The collection $\{ m_z \}_{z \in Z}$ $\a$-adically generates $M$.
\end{enumerate}
\end{thm}

\begin{proof}
Let
$\phi : \opn{F}_{\mrm{dec}}(Z, A) \to M$
be the homomorphism corresponding to the function $f : Z \to M$,
$f(z) := m_z$, as in Corollary \ref{cor:12}.
Then $\{ m_z \}_{z \in Z}$
$\a$-adically generates $M$ if and only if $\phi$ is surjective.

For every $i \geq 0$ let $M_i := A_i \otimes_A M$.
There is a commutative diagram
\begin{equation} \label{eqn:15}
\UseTips \xymatrix @C=5ex @R=5ex {
\opn{F}_{\mrm{dec}}(Z, A)
\ar[r]^(0.65){\phi}
\ar[d]
& 
M
\ar[d]
\\
\opn{F}_{\mrm{fin}}(Z, A_i)
\ar[r]^(0.65){\phi_i}
&
M_i
} 
\end{equation}
in which the vertical arrows are the surjections coming from the ring
homomorphisms
$A \to A_i$. For $i = 0$ we have
$\pi_0(m_z) = \phi_0(\delta_z) \in M_0$ for all $z \in Z$. Hence the collection 
$\{ \pi_0(m_z) \}_{z \in Z}$ generates the $A_0$-module $M_0$
if and only if $\phi_0$ is surjective. 
The implication (ii) $\Rightarrow$ (i) is now clear.

Now let us assume (i), namely that $\phi_0$ is surjective.
Since for every $i$ the ideal $\a / \a^{i+1} = \opn{Ker}(A_i \to A_0)$ is
nilpotent, the usual Nakayama Lemma (see \cite[Corollary II.3.1]{CA})
says that $\phi_i$ is surjective. 
Consider the commutative diagram
\[ \UseTips \xymatrix @C=9ex @R=5ex {
\opn{F}_{\mrm{dec}}(Z, A)
\ar[r]^(0.6){\phi}
\ar[d]
& 
M
\ar[d]^{\tau_M}
\\
\lim_{\leftarrow i}\, \opn{F}_{\mrm{fin}}(Z, A_i)
\ar[r]^(0.55){\lim_{\leftarrow i}\, \phi_i}
&
\lim_{\leftarrow i}\, M_i
} \]
gotten as the inverse limit of the sequences (\ref{eqn:15}). 
As in the proof of Proposition \ref{prop:10} one shows that the homomorphism
$\lim_{\leftarrow i}\, \phi_i$ is surjective. By Theorem \ref{thm:7}
the left vertical arrow is bijective; and by assumption $\tau_M$ is bijective.
It follows that $\phi$ is surjective.
\end{proof}

To end this section here are some remarks.

\begin{rem} \label{rem:3}
Suppose $A$ is complete.
There is a canonical pairing
\[ \opn{F}(Z, A) \times \opn{F}_{\mrm{dec}}(Z, A) \to 
A , \]
namely
\[ \bra{f, g} := \sum_{z \in Z} \, f(z)  g(z) . \]
If we put the discrete topology on 
$\opn{F}_{\mrm{dec}}(Z, A)$, and a suitable topology on 
$\opn{F}_{}(Z, A)$, then this becomes a perfect pairing
(i.e.\ it identifies each of these $A$-modules with the continuous
dual of the other).
 
Suppose $h : Y \to Z$ is a function. Then there is a ring
homomorphism
\[ h^* : \opn{F}_{}(Z, A) \to \opn{F}_{}(Y, A)  \]
called pullback, namely $h^*(f) = f \circ h$.
And there is an 
$\opn{F}_{}(Z, A)$-module homomorphism 
\[ h_* : \opn{F}_{\mrm{dec}}(Y, A) \to \opn{F}_{\mrm{dec}}(Z, A) , \]
which is
\[ h_*(g)(z) := \bosum_{y \in h^{-1}(z)} g(y) \in A . \]

In this way $\opn{F}_{\mrm{dec}}(Z, A)$ resembles the space 
$\mrm{L}^1(Z)$ from functional analysis, and $\opn{F}_{}(Z, A)$
resembles the space $\mrm{L}^{\infty}(Z)$.
\end{rem}

\begin{rem}
Suppose $\{ M_z \}_{z \in Z}$ is a collection of $\a$-adically separated
$A$-modules. By an obvious generalization of Definition
\ref{dfn:1}, we can form the decaying direct product
$\prod^{\mrm{dec}}_{z \in Z} M_z$, which is a submodule of
$\prod_{z \in Z} M_z$. 
In case $A$ is noetherian and complete, and all the modules $M_z$ are finitely
generated, one can show (just as in Theorem \ref{thm:2}) that the module 
$\prod^{\mrm{dec}}_{z \in Z} M_z$ 
is $\a$-adically complete, and it is flat if and only if all the modules
$M_z$ are flat.
\end{rem}

\section{Noetherian Rings and their Completions}
\label{sec:noeth}

In this section $A$ is a noetherian ring, and $\a$ is an ideal in it. 
The $\a$-adic completion of $A$ is $\what{A}$, and we write
$\what{\a} := \a \what{A}$, which is an ideal in $\what{A}$.
It is well-known that the ring $\what{A}$ is $\what{\a}$-adically complete,
flat over $A$, and for every $i \geq 0$ the canonical homomorphism
$A_i = A / \a^{i+1} \to \what{A} / \what{\a}^{i+1}$ is bijective.
It is also well-known that every finitely generated $\what{A}$-module is
$\what{\a}$-adically complete. We are of course allowing the case
$A = \what{A}$. See \cite[Section 10]{AM} or \cite[Section III.3]{CA}.

Let $M$ be an $\what{A}$-module. Since $\a^i M = \what{\a}^i M$ for every 
$i \geq 0$, it follows that 
$\Lambda_{\a} M = \Lambda_{\what{\a}} M$. 
So $M$ is $\a$-adically separated (resp.\ complete) if and only if it is
$\what{\a}$-adically separated (resp.\ complete). And when $M$ is separated
we have
$\opn{ord}_{\a, M} = \opn{ord}_{\what{\a}, M}$, so a
function $f : Z \to M$ is $\a$-adically decaying if and only if it is
$\what{\a}$-adically decaying.

\begin{lem} \label{lem:1}
Suppose $M$ is a finitely generated $\what{A}$-module, and $N$ is an
$\what{A}$-submodule of $M$. Then
\[ \opn{F}_{\mrm{dec}}(Z, N) = 
\opn{F}_{\mrm{dec}}(Z, M) \cap \opn{F}_{}(Z, N)  \]
as submodules of $\opn{F}(Z, M)$.
\end{lem}

\begin{proof}
Since 
$\a^{i} N \subset \a^{i} M$
for any $i \geq 0$, it follows that 
$\opn{ord}_{\a, N}(n) \leq \opn{ord}_{\a, M}(n)$
for any $n \in N$. 
By the Artin-Rees Lemma  (cf.\ \cite[Corollary III.3.1]{CA}), there is
some $i_0$ such that 
\[ N \cap (\a^{i_0 + i} M) \subset \a^{i} N \]
for all $i \geq 0$. Therefore for $n \in N$ we have
\[ \opn{ord}_{\a, M}(n) \leq \opn{ord}_{\a, N}(n) + i_0  . \]
We conclude that the $\a$-adic decay conditions with respect to $M$
and to $N$ are equivalent, for a function $f : Z \to N$. 
\end{proof}

Let us denote by $\cat{Mod}_{\mrm{f}} \what{A}$ the full subcategory of
$\cat{Mod} \what{A}$ consisting of finitely generated $\what{A}$-modules. The
subcategory $\cat{Mod}_{\mrm{f}} \what{A}$ is abelian (since $\what{A}$ is
noetherian). Note that 
$\cat{Mod}_{\mrm{f}} \what{A} \subset \cat{Mod}_{\mrm{sep}} \what{A}$.

\begin{lem} \label{lem:2}
For a given set $Z$, the functor 
\[ \opn{F}_{\mrm{dec}}(Z, -) : 
\cat{Mod}_{\mrm{f}} \what{A} \to \cat{Mod} \what{A} \]
is exact.
\end{lem}

\begin{proof}
Consider an exact sequence
\[ 0 \to M' \xar{\phi} M \xar{\psi} M'' \to 0 \]
of finitely generated $\what{A}$-modules. We want to show that the sequence
\[ 0 \to \opn{F}_{\mrm{dec}}(Z, M') \xar{\phi} 
\opn{F}_{\mrm{dec}}(Z, M) \xar{\psi} 
\opn{F}_{\mrm{dec}}(Z, M'') \to 0 \]
is also exact. 

Since $\psi(M) = M''$, it follows that  
$\psi(\a^i M) = \a^i M''$ for all $i$.
Take any 
$f \in \opn{F}_{\mrm{dec}}(Z, M'')$.
For any $z \in Z$ we can lift $f(z) \in M''$ to some element 
$g(z) \in M$, such that
$\opn{ord}_{\a, M}(g(z)) \geq \opn{ord}_{\a, M''}(f(z))$.
We get a decaying function $g : Z \to M$ lifting $f$.
So we have exactness at $\opn{F}_{\mrm{dec}}(Z, M'')$. 

Exactness at $\opn{F}_{\mrm{dec}}(Z, M)$ is by Lemma \ref{lem:1}, and 
exactness at $\opn{F}_{\mrm{dec}}(Z, M')$ is trivial.
\end{proof}

\begin{lem} \label{lem:3}
Let $M$ be a finitely generated $\what{A}$-module. Then the canonical
homomorphism
\[ M \otimes_{\what{A}} \opn{F}_{\mrm{dec}}(Z, \what{A}) \to 
\opn{F}_{\mrm{dec}}(Z, M)  \]
is bijective. 
\end{lem}

\begin{proof}
We use the standard trick of finite free presentations. 
Choose some finite presentation of $M$; namely an exact sequence
$Q \to P \to M \to 0$, where $P$ and $Q$ are finitely generated free 
$\what{A}$-modules. There is an induced commutative diagram
\[ \UseTips \xymatrix @C=5ex @R=5ex {
Q \otimes_{\what{A}} \opn{F}_{\mrm{dec}}(Z, \what{A})
\ar[r]
\ar[d]^{\phi_Q}
&
P \otimes_{\what{A}} \opn{F}_{\mrm{dec}}(Z, \what{A})
\ar[r]
\ar[d]\ar[d]^{\phi_P}
&
M \otimes_{\what{A}} \opn{F}_{\mrm{dec}}(Z, \what{A})
\ar[r]
\ar[d]\ar[d]^{\phi_M}
&
0
\\
\opn{F}_{\mrm{dec}}(Z, Q)
\ar[r]
&
\opn{F}_{\mrm{dec}}(Z, P)
\ar[r]
&
\opn{F}_{\mrm{dec}}(Z, M)
\ar[r]
&
0 \ . 
} \]
The top row is exact because of right-exactness of the tensor product; and the
bottom row is exact by Lemma \ref{lem:2}.
The homomorphisms $\phi_P$ and $\phi_Q$ are bijective since $P$ and $Q$ are
finite rank free modules. It follows that $\phi_M$ is also bijective.
\end{proof}

Here is the main result of this section. Observe that it refers only to the
complete ring $\what{A}$.

\begin{thm} \label{thm:2}
Let $\what{A}$ be a noetherian ring, $\what{\a}$-adically complete
with respect to some ideal $\what{\a}$. Let $Z$ be any set.
Then\tup{:}
\begin{enumerate}
\item For any $i \geq 0$ the canonical homomorphism
\[ A_i \otimes_{\what{A}} \opn{F}_{\mrm{dec}}(Z, \what{A}) \to 
\opn{F}_{\mrm{fin}}(Z, A_i) \]
is bijective. Here $A_i := \what{A} / \what{\a}^{i+1}$.
\item The $\what{A}$-module $\opn{F}_{\mrm{dec}}(Z, \what{A})$ is flat and
$\what{\a}$-adically complete. 
\end{enumerate}
\end{thm}

\begin{proof}
(1) This is true by Lemma  \ref{lem:3}, with $M := A_i$.

\medskip \noindent
(2) Since $\what{A}$ is noetherian, the $\what{A}$-module 
$\opn{F}_{\mrm{dec}}(Z, \what{A})$ is flat if and only if the functor
$- \otimes_{\what{A}} \opn{F}_{\mrm{dec}}(Z, \what{A})$ is exact on 
$\cat{Mod}_{\mrm{f}} \what{A}$. The latter is true by Lemmas \ref{lem:2} and
\ref{lem:3}.

As for completeness, combining part (1) above with Theorem \ref{thm:7}
(for the module $M := \what{A}$) we see that the canonical homomorphism
\[ \tau_{\opn{F}_{\mrm{dec}}(Z, \what{A})} :
\opn{F}_{\mrm{dec}}(Z, \what{A}) \to 
\lim_{\leftarrow i}\, \bigl( A_i \otimes_{\what{A}} 
\opn{F}_{\mrm{dec}}(Z, \what{A}) \bigr) \]
is bijective.
\end{proof}

Here are several corollaries to Theorem \ref{thm:2}.

\begin{cor} \label{cor:3}
Let $M$ be any $A$-module. Its $\a$-adic completion $\what{M}$
is $\a$-adically complete.
\end{cor}

\begin{proof}
Choose any surjection 
$\phi : \opn{F}_{\mrm{fin}}(Z, A) \to M$, where $Z$ is some set, and
write $Q := \opn{F}_{\mrm{fin}}(Z, A)$.
By Proposition \ref{prop:10} the homomorphism
$\what{\phi} : \what{Q} \to \what{M}$
is surjective. Hence for every $i \geq 0$ we get a commutative diagram
\[ \UseTips \xymatrix @C=8ex @R=5ex {
A_i \otimes_A Q
\ar[r]^{\tau_{Q, i}}
\ar[d]_{\opn{id}_{A_i} \otimes\, \phi}
& 
A_i \otimes_A \what{Q}
\ar[d]
\ar[d]^{\opn{id}_{A_i} \otimes\, \what{\phi}}
\\
A_i \otimes_A M
\ar[r]^{\tau_{M, i}}
&
A_i \otimes_A \what{M}
} \]
with surjective vertical arrows. 
By Corollary \ref{cor:7} and Theorem \ref{thm:2}(2) the module $\what{Q}$ is
$\a$-adically complete, and hence
by Theorem \ref{thm:6} the homomorphisms $\tau_{Q, i}$ is surjective.
It follows that $\tau_{M, i}$ is also surjective, for every $i$. 
Again using Theorem \ref{thm:6} we conclude that $\what{M}$ is complete.
\end{proof}

\begin{cor} \label{cor:16}
Let $B$ be a ring, and let $\mfrak{b}$ be a finitely generated ideal in it.
Given any $B$-module $M$, its $\mfrak{b}$-adic completion
$\Lambda_{\mfrak{b}} M$ is $\mfrak{b}$-adically complete. 
\end{cor}

\begin{proof}
Choose generators $b_1, \ldots, b_n$ of the ideal $\mfrak{b}$.
Consider the polynomial ring
$A := \mbb{Z}[t_1, \ldots, t_n]$, the ideal 
$\a := (t_1, \ldots, t_n) \subset A$, and the ring homomorphism
$f : A \to B$ defined by $f(t_i) := b_i$.
Then for any $B$-module $N$ there is a canonical isomorphism of $B$-modules
 $\Lambda_{\mfrak{b}} N \cong \Lambda_{\a} N$,
that commutes with the homomorphisms from $N$. Since $A$ is noetherian we know
from Corollary \ref{cor:3} that 
$N := \Lambda_{\a} M$ is $\a$-adically complete. 
\end{proof}

\begin{rem} \label{rem:2}
The assertions of Corollaries \ref{cor:3} and \ref{cor:16} are not new, yet
they seem to be virtually unknown. After we proved Theorem \ref{thm:2},
A.-M. Simon mentioned to us the book \cite{St}, and in Subsection 2.2.5 of that
book we located these assertions (in slightly different wording). We then
learned that Corollary \ref{cor:16} appeared much earlier as
\cite[Theorem 15]{Ml}. Note that our proof of Theorem \ref{thm:2},
involving the concept of decaying functions, is completely
new, and is not similar to the proofs in these cited works. 

Corollary \ref{cor:16} resembles \cite[Proposition III.14]{CA}. However a close
inspection reveals that these two assertions refer to distinct notions of
completion. See Example \ref{exa:7}, Corollary \ref{cor:18} and the 
discussion between them.
\end{rem}

\begin{cor} \label{cor:14}
Let $M$ be any $A$-module. Then 
the $A$-module $\opn{F}_{\mrm{dec}}(Z, \what{M})$ is $\a$-adically complete.
\end{cor}

\begin{proof}
According to Corollary \ref{cor:3} the module $\what{M}$ is complete. 
By Corollary \ref{cor:7} we know that $\opn{F}_{\mrm{dec}}(Z, \what{M})$ is 
(canonically isomorphic to) the
$\a$-adic completion of $\opn{F}_{\mrm{fin}}(Z, M)$. Now use
Corollary \ref{cor:3} again to conclude that 
$\opn{F}_{\mrm{dec}}(Z, \what{M})$ is complete.
\end{proof}

\begin{cor}  \label{cor:11}
Let $Z$ be a set, let $M$ be an $\a$-adically complete $A$-module, and let
$f : Z \to M$ be any function. Then there is a
unique $A$-linear homomorphism
\[ \phi : \opn{F}_{\mrm{dec}}(Z, \what{A}) \to M \]
such that 
$\phi(\delta_z) = f(z)$
for all $z \in Z$. 
\end{cor}

\begin{proof}
The existence of such a homomorphism was already proved in Corollary
\ref{cor:12}. Recall that the formula is
\[ \phi(g) = \sum_{z \in Z} g(z) f(z) \in M \]
for $g \in \opn{F}_{\mrm{dec}}(Z, \what{A})$.
Uniqueness is because $M$ is complete,
and the image of $\opn{F}_{\mrm{fin}}(Z, A)$ in
$\opn{F}_{\mrm{dec}}(Z, \what{A})$, which is the $A$-submodule generated by the 
collection \linebreak $\{ \delta_z \}_{z \in Z}$, is dense in 
$\opn{F}_{\mrm{dec}}(Z, \what{A})$, by Theorem \ref{thm:2}(1).
\end{proof}

\begin{exa}
Take any set $Z$.  Consider the function 
$f : Z \to \opn{F}_{\mrm{dec}}(Z, \what{A})$,
$f(z) := \delta_z$. The corresponding homomorphism $\phi$ is the identity of
$\opn{F}_{\mrm{dec}}(Z, \what{A})$. This says that
\[ g = \sum_{z \in Z} g(z) \delta_z \]
for any $g \in \opn{F}_{\mrm{dec}}(Z, \what{A})$.
\end{exa}

\begin{dfn} 
An $A$-module $P$ is called {\em $\a$-adically free} if it
isomorphic to the $A$-module $\opn{F}_{\mrm{dec}}(Z, \what{A})$
for some set $Z$.
\end{dfn}

\begin{cor}
Suppose $B$ is another noetherian ring, $\mfrak{b} \subset B$
is an ideal, and $f : A \to B$ is a ring homomorphism
satisfying $f(\a) \subset \mfrak{b}$. If $P$ is an $\a$-adically free
$A$-module, then the $B$-module
\[ B \hatotimes{A} P := \Lambda_{\mfrak{b}} (B \otimes_A P) \]
is $\mfrak{b}$-adically free.
\end{cor}

\begin{proof}
Letting $B_i := B / \mfrak{b}^{i+1}$, we have induced ring
homomorphisms $A_i \to B_i$ for all $i \geq 0$. 
Choose an $A$-module isomorphism
$P \cong \opn{F}_{\mrm{dec}}(Z, \what{A})$. 
Then by Theorem \ref{thm:2} we have canonical isomorphisms
\[ B_i \otimes_A P \cong B_i \otimes_{A_i} A_i \otimes_A 
\opn{F}_{\mrm{dec}}(Z, \what{A})
\cong B_i \otimes_{A_i} \opn{F}_{\mrm{fin}}(Z, A_i)
\cong \opn{F}_{\mrm{fin}}(Z, B_i) . \]
We see that
\[ \Lambda_{\mfrak{b}} (B \otimes_A P) \cong 
\lim_{\leftarrow i}\, (B_i \otimes_A P) \cong
\lim_{\leftarrow i}\, \opn{F}_{\mrm{fin}}(Z, B_i) \cong
\opn{F}_{\mrm{dec} / \mfrak{b}}(Z, \Lambda_{\mfrak{b}} B) , \]
where $\opn{F}_{\mrm{dec} / \mfrak{b}}(Z, -)$ refers to the 
$\mfrak{b}$-adic decay condition.
\end{proof}

\begin{prop} \label{prop:11}
The following two conditions are equivalent for an $A$-module $P$\tup{:}
\begin{enumerate}
\rmitem{i} $P$ is $\a$-adically free.
\rmitem{ii} $P$ is isomorphic to $\a$-adic completion $\what{Q}$ of some free
$A$-module $Q$.
\end{enumerate}
\end{prop}

\begin{proof}
First suppose 
$P \cong \what{Q}$ for some free $A$-module $Q$. By choosing a basis for
$Q$, indexed by a set $Z$, we get an isomorphism
$Q \cong \opn{F}_{\mrm{fin}}(Z, A)$.
According to Corollary \ref{cor:7} we get an isomorphism
$P \cong \opn{F}_{\mrm{dec}}(Z, \what{A})$.
The reverse implication is proved similarly.
\end{proof}

\begin{exa}
Suppose $A$ is complete, $\K$ is a field, and $\K \to A$ is a ring homomorphism.
Let $V$ be a $\K$-module. The $A$-module
$A \otimes_{\K} V$ is free, and therefore its 
$\a$-adic completion 
$A \hatotimes{\K} V := \Lambda_{\a} (A \otimes_{\K} V)$ is $\a$-adically free.
\end{exa}

\begin{cor} \label{cor:6}
Suppose $M$ is an $\a$-adically complete $A$-module. Then 
there is a surjection $\phi : P \to M$ for some $\a$-adically free $A$-module
$P$.
\end{cor}

\begin{proof}
Choose a surjection $\psi : Q \to M$, where $Q$ is some free $A$-module.
By Proposition \ref{prop:10} the induced homomorphism 
$\what{\psi} : \what{Q} \to \what{M}$
is surjective. We know that $P := \what{Q}$ is $\a$-adically free
(see Proposition \ref{prop:11}), and that 
$\tau_M : M \to \what{M}$ is bijective. So we can take
$\phi := \tau_M^{-1} \circ \what{\psi}$.
\end{proof}

\begin{dfn}
An $A$-module $P$ is called {\em $\a$-adically projective} if it satisfies the
following two conditions:
\begin{enumerate}
\rmitem{i}$P$ is $\a$-adically complete.

\rmitem{ii} Suppose $M$ and $N$ are $\a$-adically complete $A$-modules, 
and $\phi : M \to N$ a surjective homomorphism. Then any homomorphism
$\psi : P \to N$ can be lifted to a homomorphism 
$\til{\psi} : P \to M$,
such that $\psi = \phi \circ \til{\psi}$.
\end{enumerate}
\end{dfn}

\begin{rem}
It can be shown that condition (ii) above is equivalent to $P$ being {\em
topologically projective}, in the sense of 
\cite[Section 0${}_{\mrm{IV}}$.19.2]{EGA-IV}
\end{rem}

\begin{cor} \label{cor:15}
An $A$-module $P$ is $\a$-adically projective if and only if it is a direct
summand of an $\a$-adically free module.
\end{cor}

\begin{proof}
First assume that $P$ is a direct summand of an $\a$-adically free module; say
$P \oplus P' = Q$. By Theorem \ref{thm:2}(2) and Corollary \ref{cor:11}, the 
$\a$-adically free module $Q$ is $\a$-adically projective.
And it is easy to see that a direct summand of an $\a$-adically projective
module is also $\a$-adically projective. 

Conversely, assume that $P$ is $\a$-adically projective. Because $P$ is
complete, by Corollary \ref{cor:6}  there exists a surjection
$\phi : Q \to P$ for some $\a$-adically free module $Q$. 
Since $P$ and $Q$ are both complete, condition (ii) says that
$\phi$ is split. 
\end{proof}

To finish this section, here are a couple of examples and a remark. The first
example is a bit facile, but instructive.

\begin{exa} \label{exa:6}
Let $\K$ be a field, and let
$A := \K[[ t ]]$, the ring of formal power series in one variable. 
It is a complete noetherian local ring with maximal ideal $\a = (t)$.
Let $K := \K((t))$, the field of fractions of $A$. Consider the inclusion
$\phi : A \to K$. For any $i \geq 0$ we have
$A_i \otimes_A K = 0$. Therefore $\what{K} = 0$, and 
$\what{\phi} : \what{A} \to \what{K}$ is not injective.
\end{exa}

We see that the functor $\Lambda_{\a}$ does not respect injections. Since it
does respect surjections (Proposition \ref{prop:10}), one is tempted to guess
that $\Lambda_{\a}$ is right exact. But here is a counterexample.

\begin{exa} \label{exa:5}
With $A := \K[[t]]$ and $\a = (t)$ as in the previous example, let 
$P, Q := \opn{F}_{\mrm{dec}}(\N, A)$. Define a homomorphism
$\phi : P \to Q$ by 
$\phi(\delta_i) := t^i \delta_i$, where 
$\delta_i \in \opn{F}_{\mrm{dec}}(\N, A)$
are the delta functions. It is easy to see that $\phi$ is injective. 

We claim that the submodule $L := \opn{Im}(\phi)$ is not closed in $Q$. 
Indeed, consider the element 
$f := \sum_{i \in \N} t^i \delta_i \in Q$. 
Clearly $f$ is in the closure $\bar{L}$ of $L$. If there were
some $g \in P$ such that $f = \phi(g)$, then writing
$a_i := g(i) \in A$, we would have
$g = \sum_{i} a_i \delta_i$. Hence
\[ f = \phi(g) = \sum_{i} a_i \phi(\delta_i) = 
\sum_{i} a_i t^i \delta_i . \]
By uniqueness of the series expansion, it would follow that 
$a_i = 1$ for all $i$. But then the function 
$g : \N \to A$ would not be decaying; so we arrive at a contradiction.

Let us define $M := Q / L$.  So there is an exact sequence of $A$-modules
\begin{equation} \label{eqn:6}
0 \to P \xar{\phi} Q \xar{\psi} M \to  0 . 
\end{equation}
Now $P$ and $Q$ are complete, so we can identify them with their completions 
$\what{P}$ and $\what{Q}$. According to Proposition \ref{prop:10} the
homomorphism $\what{\psi} : Q \to \what{M}$ is surjective, and by Corollary
\ref{cor:3} the module
$\what{M}$ is complete. Therefore 
$\opn{Ker}(\what{\psi}) = \bar{L}$. Because $L \subsetneq \bar{L}$ we see
that  $\tau_M : M \to \what{M}$ is surjective but not bijective. Thus
$M$ {\em is not $\a$-adically complete}.
Also, since $\what{\phi} = \phi$, we see that the sequence 
\[ 0 \to \what{P} \xar{\what{\phi}} \what{Q} \xar{\what{\psi}} \what{M} \to  0 
\]
that we get by completing (\ref{eqn:6}) is not exact at 
$\what{Q}$. This shows that {\em the functor $\Lambda_{\a}$ is not right exact}.
\end{exa}

\begin{rem} \label{rem:1}
Suppose $A$ is complete and $Q$ is a free $A$-module of {\em countable} rank.
V. Drinfeld and M. Hochster mentioned to us an alternative proof of the fact
that $\what{Q}$ is flat and $\a$-adically complete.  In this
case $Q$ is isomorphic, as $A$-module, to the polynomial algebra $A[t]$.
Then the completion $\what{Q}$ is isomorphic, as $A$-module, to the algebra
$A \{ t \}$ of restricted formal power series; see Example
\ref{exa:2}. It is shown in \cite{CA} that 
$A \{ t \}$ is $\a$-adically complete and flat over $A$.
\end{rem}

\section{Complete Noetherian Local Rings}
\label{sec:local}

In this section $A$ is a complete noetherian local commutative ring, with
maximal ideal $\m$.
For $i \geq 0$ we write $A_i := A / \m^{i+1}$. 

\begin{dfn} \label{dfn:3}
An {\em $\m$-adic system of $A$-modules} is a collection 
$\{ M_i \}_{i \in \mbb{N}}$ of \linebreak $A$-modules, together with a
collection $\{ \psi_i \}_{i \in \mbb{N}}$ of 
homomorphisms 
$\psi_i : M_{i+1} \to M_i$.
The conditions are:
\begin{enumerate}
\rmitem{i} For every $i$ one has $\m^{i+1} M_i = 0$. Thus 
$M_i$ is an $A_i$-module.
\rmitem{ii} For every $i$ the $A_i$-linear homomorphism
$A_i \otimes_{A_{i+1}} M_{i+1} \to M_i$
induced by $\psi_i$ is an isomorphism. 
\end{enumerate}
\end{dfn}

Usually the collection of homomorphisms 
$\{ \psi_i \}_{i \in \mbb{N}}$ remains implicit.

\begin{exa}
Suppose $M$ is an $A$-module, and let $M_i := A_i \otimes_A M$. Then 
$\{ M_i \}_{i \in \mbb{N}}$ is an $\m$-adic system of
$A$-modules. 
\end{exa}

\begin{thm} \label{thm:1}
Let $A$ be a complete noetherian local ring, with maximal ideal
$\m$, and let $\{ M_i \}_{i \in \mbb{N}}$
be an $\m$-adic system of $A$-modules. Assume that 
$M_i$ is flat over $A_i$ for every $i$. Define
$M := \lim_{\leftarrow i}\, M_i$.
Then the following hold.
\begin{enumerate}
\item The $A$-module $M$ is $\m$-adically free.
\item For every $i \geq 0$ the canonical homomorphism
$A_i \otimes_A M \to M_i$
is bijective.
\end{enumerate}
\end{thm}

We need an auxiliary result. 

\begin{lem}
In the setup of the theorem, suppose $M_i$ is a free 
$A_i$-module with basis 
$\{ \bar{m}_z \}_{z \in Z}$. Let $m_z \in M_{i+1}$ be a lifting of 
$\bar{m}_z$. Then $M_{i+1}$ is a free $A_{i+1}$-module with basis
$\{ m_z \}_{z \in Z}$.
\end{lem}

This result must be  well known, but we could not locate a reference
in the literature. The closest we got is \cite[Proposition 3.G]{Ma}.

\begin{proof}
Since the ideal 
$\m^{i+1} / \m^{i+2} = \opn{Ker}(A_{i+1} \to A_i)$ is nilpotent, 
Nakayama's Lemma says that the collection 
$\{ m_z \}_{z \in Z}$ generates $M_{i+1}$. So there is an exact
sequence of $A_{i+1}$-modules
\[ 0 \to N \to \opn{F}_{\mrm{fin}}(Z, A_{i+1}) \xar{\phi}
M_{i+1} \to 0 , \]
where $\phi(\delta_z) := m_z$ and $N := \opn{Ker}(\phi)$. Applying the
operation 
$A_i \otimes_{A_{i+1}} -$ to this sequence we get an exact sequence
\[ \opn{Tor}^{A_{i+1}}_1(A_i, M_{i+1}) \to 
A_i \otimes_{A_{i+1}} N \to 
\opn{F}_{\mrm{fin}}(Z, A_{i}) \xar{\bar{\phi}}
M_{i} \to 0 . \] 
Since $M_{i+1}$ is flat we get
$\opn{Tor}^{A_{i+1}}_1(A_i, M_{i+1}) = 0$.
On the other hand, since $\{ \bar{m}_z \}_{z \in Z}$ is a basis, we
see that  
$\bar{\phi} : \opn{F}_{\mrm{fin}}(Z, A_{i}) \to M_{i}$
is bijective. It follows that
$A_i \otimes_{A_{i+1}} N = 0$. Using the Nakayama Lemma once more we
see that $N = 0$.
\end{proof}

\begin{proof}[Proof of the theorem]
Since $A_0$ is a field, the $A_0$-module $M_0$ is free. Let us
choose a basis $\{ m_z \}_{z \in Z}$ for $M_0$.
By the lemma above, used recursively, we can lift this basis
to a basis of $M_i$ for every $i \geq 0$. Thus we get an inverse
system of isomorphisms 
$M_i \cong \opn{F}_{\mrm{fin}}(Z, A_{i})$. 
In the limit we get
$M \cong \opn{F}_{\mrm{dec}}(Z, A)$, by Theorem \ref{thm:2}(1,2).
So $M$ is $\m$-adically free.

Finally, according to Theorem \ref{thm:2}(1) we have 
$A_i \otimes_A M \cong M_i$.
\end{proof}

\begin{cor} \label{cor:1}
The following conditions are equivalent for an $A$-module $M$\tup{:}
\begin{enumerate}
\rmitem{i} $M$ is flat and $\m$-adically complete. 
\rmitem{ii} There is an $\m$-adic system of
$A$-modules $\{ M_i \}_{i \in \mbb{N}}$, such that $M_i$ is flat over
$A_i$ for every $i$, and an isomorphism of $A$-modules
$M \cong \lim_{\leftarrow i}\, M_i$.
\rmitem{iii} $M$ is $\m$-adically free.
\end{enumerate}
\end{cor}

\begin{proof}
The implication (i) $\Rightarrow$ (ii) is trivial. 
The implication (ii) $\Rightarrow$ (iii) is Theorem \ref{thm:1}(1).
And the implication (iii) $\Rightarrow$ (i) is Theorem \ref{thm:2}(2).
\end{proof}

\begin{rem}
A special case of Corollary \ref{cor:1}, namely when 
$A = \K[[t]]$, the ring of formal power series in a variable
$t$ over a field $\K$, was proved in \cite[Lemma A.1]{CFT}.
\end{rem}

\begin{rem}
Assume $A$ is an equal characteristic complete local ring, namely
it contains a field $\K$ such that $\K \cong A / \m$.
Let $Q$ be a free $A$-module and $P := \what{Q}$. In this case there
is an alternative way to prove Theorem \ref{thm:2}(2). 
First choose an isomorphism $Q \cong A \otimes_{\K} V$ for some
$\K$-module $V$. Next choose a filtered $\K$-basis 
$\{ a_j \}_{j \in \mbb{N}}$ for $A$ (cf.\ \cite[Definition 6.5]{Ye1}; we may
assume $\m$ is not nilpotent). Then we obtain $\K$-module
isomorphisms
$A \cong \prod_{j \geq 0} \K$,
$P \cong \prod_{j \geq 0} V$ and
$\m^i P \cong \prod_{j \geq j_i} V$,
where 
$0 = j_0 < j_1 < j_2 \cdots$.
This implies completeness of $P$. Flatness is proved similarly, but
it is a bit more complicated.
\end{rem}

\section{Flat Complete Sheaves of Modules}
\label{sec:sheaves}

In this section there is some overlap with material from \cite{KS}.

Let $X$ be a topological space and $A$ a commutative ring.
Recall that given sheaves $\mcal{M}, \mcal{N}$ of $A$-modules on $X$, the
sheaf of $A$-modules
$\mcal{N} \otimes_A \mcal{M}$ is the sheaf associated to the presheaf
\[ U \mapsto \Gamma(U, \mcal{N}) \otimes_A \Gamma(U, \mcal{M}) , \]
for open sets $U \subset X$. 
If $N$ is an $A$-module, then we can similarly consider the sheaf 
$N \otimes_A \mcal{M}$ on $X$; this is the sheaf associated to the presheaf
\[ U \mapsto N \otimes_A \Gamma(U, \mcal{M}) . \]
Given an $A$-algebra $B$, the sheaf
$B \otimes_A \mcal{M}$ becomes a sheaf of $B$-modules.
If \linebreak  $\{ \mcal{M}_i \}_{i \in \N}$ is an inverse system of sheaves of
modules on $X$, then 
$\opn{lim}_{\leftarrow i}\, \mcal{M}_i$ is the sheaf
$U \mapsto \opn{lim}_{\leftarrow i}\, \Gamma(U, \mcal{M}_i)$.
Recall that the sheaf $\mcal{M}$ is said to be flat over $A$ if for
every point $x \in X$ the stalk $\mcal{M}_x$ is a flat $A$-module. 

Now suppose $A$ is a complete noetherian local ring, with maximal
ideal $\m$. For $i \geq 0$ we write $A_i := A / \m^{i+1}$. 

\begin{dfn}
Let $\mcal{M}$ be a sheaf of $A$-modules on $X$. 
\begin{enumerate}
\item The $\m$-adic completion of $\mcal{M}$ is the sheaf
\[ \what{\mcal{M}} := \opn{lim}_{\leftarrow i}\, 
(A_i \otimes_A \mcal{M}) . \]
\item The sheaf $\mcal{M}$ is called {\em $\m$-adically complete} if
the canonical sheaf homomorphism
$\tau_{\mcal{M}} : \mcal{M} \to \what{\mcal{M}}$ is an isomorphism.
\end{enumerate}
\end{dfn}

We sometimes use the notation
$\Lambda_{\m} \mcal{M} := \what{\mcal{M}}$. With this notation we have an
additive functor
\[ \Lambda_{\m} : \cat{Mod} A_X \to \cat{Mod} A_X . \]
Here $A_X$ is the constant sheaf $A$ on $X$, and $\cat{Mod} A_X$ is the
category of sheaves of $A_X$-modules, which is the same as the category of
sheaves of $A$-modules on $X$.

Suppose $B$ is another complete noetherian local ring, with maximal ideal $\n$,
and we are given a local homomorphism $A \to B$. For any sheaf of $A$-modules
$\mcal{M}$ on $X$, and any $B$-module $N$, we write
\[ N \hatotimes{A} \mcal{M} := \Lambda_{\n} (N \otimes_{A} \mcal{M}) . \]

The inverse limit in the completion operation does not commute with
the direct limit of passing to stalks. Hence the stalk $\mcal{M}_x$ of
an $\m$-adically complete sheaf of $A$-modules $\mcal{M}$, at a point
$x \in X$, is usually not an $\m$-adically complete $A$-module. 
This is a well known fact; see \cite[Paragraph 10.1.5]{EGA-I}, or
the next example.

\begin{exa}
Take $X := \mbf{A}^1_{\K} = \opn{Spec} \K[t]$, the 
affine line over an infinite field $\K$, with coordinate $t$ and
structure sheaf $\mcal{O}_X$. 
Let $A := \K[[s ]]$, the formal power series ring in the variable
$s $. This is a complete noetherian local ring, whose maximal
ideal is $\m = (s)$. Let
\[ \mcal{M} := \mcal{O}_X[[s ]] \cong  
A \hatotimes{\K} \mcal{O}_X  . \]
The sheaf $\mcal{M}$ is $\m$-adically complete; indeed on any open
set $U \subset X$ (they are all affine) one has
$\Gamma(U, \mcal{M}) \cong \Gamma(U, \mcal{O}_X)[[s ]]$.

Now let us look at the closed point $x := (t) \in X$. 
Here the stalk is
\[ \mcal{M}_x \cong \opn{lim}_{U \rightarrow} \, 
\Gamma(U, \mcal{O}_X)[[s ]] , \]
where $U$ runs over the open neighborhoods of $x$. 
This is a dense submodule of its completion
$\what{\mcal{M}_x} \cong \mcal{O}_{X, x}[[s ]] \cong 
\K[t]_{(t)}[[s ]]$.
Given an element $a \in \mcal{M}_x$, there is an open
neighborhood $U$ of $x$, such that 
$a = \sum_{i \geq 0} a_i s ^i$,
with
$a_i \in \Gamma(U, \mcal{O}_X)$. 
Thus if we choose a sequence $\{ \lambda_i \}_{i \geq 0}$
of distinct elements of $\K$, all nonzero, then 
the power series
$a := \sum_{i \geq 0} (t - \lambda_i)^{-1} s ^i$
is in $\what{\mcal{M}_x}$ but not in $\mcal{M}_x$. 
\end{exa}

Even if the ideal $\m$ is nilpotent, so completion is not an issue, it
is not very useful to consider sheaves of $A$-modules on $X$ that are
locally free. This is because such a sheaf must be locally constant. 
The standard practice is to talk about flat sheaves of $A$-modules.

Let $\mcal{M}$ be a sheaf of $A$-modules on $X$. 
For $i \geq 0$ we define $\m^i \mcal{M}$ to be the image of the canonical
sheaf homomorphism
$\m^i \otimes_A \mcal{M} \to \mcal{M}$.
Let 
\[ \opn{gr}^i_{\m} \mcal{M} := \m^i \mcal{M} / \m^{i+1} \mcal{M} . \]
The direct sum 
$\opn{gr}_{\m} \mcal{M} := \boplus_i \opn{gr}^i_{\m} \mcal{M}$ is a sheaf of
graded modules over the graded ring $\opn{gr}_{\m} A$. 

\begin{prop}
Let $\mcal{M}$ be a flat sheaf of  $A$-modules on $X$.
Then the canonical sheaf homomorphism
\[ (\opn{gr}_{\m} A) \otimes_{A_0} \opn{gr}^0_{\m} \mcal{M} \to 
\opn{gr}_{\m} \mcal{M} \]
is an isomorphism. 
\end{prop}

\begin{proof}
It is enough to show that this homomorphism becomes an isomorphism 
at stalks. But at a point $x \in X$ the $A$-module $\mcal{M}_x$
is flat, so we can use \cite[Theorem III.5.1]{CA}.
\end{proof}

\begin{dfn} \label{dfn:20}
Let $\mcal{N}$ be a sheaf of abelian groups on $X$.
We denote by $\mrm{H}^1(X, \mcal{N})$ its first sheaf cohomology.
\begin{enumerate}
\item We say that an open set $U$ of $X$ is {\em $\mcal{N}$-simply connected} 
if $\mrm{H}^1(U, \mcal{N}) = 0$. 
\item The space $X$ is said to be {\em locally $\mcal{N}$-simply connected} if
it has a basis of the topology consisting of open sets
that are $\mcal{N}$-simply connected.
\end{enumerate}
\end{dfn}

\begin{exa}
Here are a few typical examples of a topological space $X$, and a
sheaf $\mcal{N}$, such that $X$ is locally $\mcal{N}$-simply connected:
\begin{enumerate}
\item $X$ is an algebraic variety over a field, with structure sheaf 
$\mcal{O}_X$, and $\mcal{N}$ is a coherent $\mcal{O}_X$-module. Then
any affine open set $U$ is $\mcal{N}$-simply connected.
\item $X$ is a complex analytic manifold, with structure sheaf 
$\mcal{O}_X$, and $\mcal{N}$ is a coherent $\mcal{O}_X$-module. Then
any Stein open set $U$ is $\mcal{N}$-simply connected.
\item $X$ is a differentiable manifold, with structure
sheaf $\mcal{O}_X$, and $\mcal{N}$ is any $\mcal{O}_X$-module. Then
any open set $U$ is $\mcal{N}$-simply connected.
\item $X$ is a differentiable manifold, and $\mcal{N}$ is a constant
sheaf of abelian groups. Then any contractible open set $U$ is $\mcal{N}$-simply
connected.
\end{enumerate}
\end{exa}

\begin{thm} \label{thm:3}
Let $A$ be a complete noetherian local ring, with maximal
ideal $\m$. Let $X$ be a topological space, and let $\mcal{M}$ be a 
flat $\m$-adically complete sheaf of $A$-modules on  $X$. 
We write $\mcal{M}_i := A_i \otimes_A \mcal{M}$,
$M :=  \Gamma(X, \mcal{M})$ and 
$M_i := \Gamma(X, \mcal{M}_i)$
for $i \geq 0$.
Assume that $X$ is $\mcal{M}_0$-simply connected. Then the following
are true.
\begin{enumerate}
\item The $A$-module $M$ is $\m$-adically free.
\item For every $i \geq 0$ the canonical homomorphism
$A_i \otimes_{A} M \to M_i$ is bijective.
\end{enumerate}
\end{thm}

We need a lemma first.

\begin{lem} \label{lem:5}
In the setup of the theorem, let $N$ be an $A_i$-module. Then\tup{:}
\begin{enumerate}
\item $\mrm{H}^1(X, N \otimes_{A_i} \mcal{M}_i) = 0$.
\item The canonical homomorphism
\[ N \otimes_{A_i} M_i \to 
\Gamma(X, N \otimes_{A_i} \mcal{M}_i) \] 
is bijective.
\end{enumerate}
\end{lem}

Again this is familiar, but we did not find a reference.

\begin{proof}
(1) The proof is by induction on $i$. For $i = 0$ the ring
$\K := A_0$ is a field, so $N$ is a free $\K$-module, and
\[ \mrm{H}^1(X, N \otimes_{\K} \mcal{M}_0) \cong
N \otimes_{\K} \mrm{H}^1(X, \mcal{M}_0) = 0 . \]
Now assume $i \geq 1$. We have an exact sequence of $A_i$-modules
\[ 0 \to V \to N \to A_{i-1} \otimes_{A_i} N \to 0 , \]
where $V$ is some $\K$-module. Since the sheaf
$\mcal{M}_i$ is flat over $A_i$, there is an exact sequence of sheaves
\[ 0 \to V \otimes_{A_i} \mcal{M}_i \to 
N \otimes_{A_i} \mcal{M}_i \to 
A_{i-1} \otimes_{A_i} N  \otimes_{A_i} \mcal{M}_i \to 0 , \]
which can be rewritten as
\[ 0 \to V \otimes_{\K} \mcal{M}_0 \to 
N \otimes_{A_i} \mcal{M}_i \to 
\bar{N}  \otimes_{A_{i-1}} \mcal{M}_{i-1} \to 0 , \]
where $\bar{N} := A_{i-1} \otimes_{A_i} N$. 
In global cohomology we get an an exact sequence
\[ \cdots \to \mrm{H}^1(X, V \otimes_{\K} \mcal{M}_0) 
\to \mrm{H}^1(X, N \otimes_{A_i} \mcal{M}_i) 
\to \mrm{H}^1(X,  \bar{N}  \otimes_{A_{i-1}} \mcal{M}_{i-1}) \to
\cdots . \]
The induction hypothesis says that the two extremes vanish; and hence
so does the middle term.

\medskip \noindent
(2) The proof is like the first part. For $i = 0$ we have
$N \otimes_{\K} M_0 \cong \Gamma(X, N \otimes_{\K} \mcal{M}_0)$
since $N$ is a free $\K$-module.  
For $i \geq 1$ we have a commutative diagram
\[ \UseTips \xymatrix @C=4ex @R=5ex {
&
V \otimes_{\K} M_0
\ar[r]
\ar[d]
& 
N \otimes_{A_i} M_i
\ar[r]
\ar[d]
&
\bar{N} \otimes_{A_{i-1}} M_{i-1}
\ar[d]
\ar[r]
& 
0
\\
0 
\ar[r]
&
\Gamma(X, V \otimes_{\K} \mcal{M}_0)
\ar[r]
&
\Gamma(X, N \otimes_{A_i} \mcal{M}_i)
\ar[r]
&
\Gamma(X, \bar{N}  \otimes_{A_{i-1}} \mcal{M}_{i-1})
} \]
with exact rows. By induction the extreme vertical arrows are
bijective. Hence so is the middle one.
\end{proof}

\begin{proof}[Proof of the theorem]
We know that 
$M \cong \opn{lim}_{\leftarrow i}\, M_i$. 
In view of Theorem \ref{thm:1} it suffices to
prove that for each $i$ the module $M_i$ is flat over $A_i$, and the canonical
homomorphism
$A_{i-1} \otimes_{A_{i}} M_{i} \to M_{i-1}$ is bijective.
The second assertion is true by Lemma \ref{lem:5}(2), taking 
$N := A_{i-1}$. 

As for flatness of $M_i$,  take an exact sequence 
\[ 0 \to N' \to N \to N'' \to 0 \]
of $A_i$-modules. Since the sheaf $\mcal{M}_i$ is flat over $A_i$, we get an 
exact sequence of sheaves
\[ 0 \to N' \otimes_{A_i} \mcal{M}_i \to 
N \otimes_{A_i} \mcal{M}_i \to N'' \otimes_{A_i} \mcal{M}_i \to 0 . \]
By Lemma \ref{lem:5}(1) we know that 
$\mrm{H}^1(X, N' \otimes_{A_i} \mcal{M}_i) = 0$, so
the sequence
\[ 0 \to \Gamma(X, N' \otimes_{A_i} \mcal{M}_i)
\to \Gamma(X, N \otimes_{A_i} \mcal{M}_i) \to 
\Gamma(X, N'' \otimes_{A_i} \mcal{M}_i) \to 0 \]
is exact. Finally, using Lemma \ref{lem:5}(2) we see that the
sequence
\[ 0 \to N' \otimes_{A_i} M_i \to 
N \otimes_{A_i} M_i \to 
N'' \otimes_{A_i} M_i \to 0 \]
is also exact.
\end{proof}

Here is the geometric version of Definition \ref{dfn:3}.

\begin{dfn}
An {\em $\m$-adic system of sheaves of $A$-modules} on $X$
is a collection \linebreak
$\{ \mcal{M}_i \}_{i \in \mbb{N}}$ of sheaves of $A$-modules, 
together with a collection 
$\{ \psi_i \}_{i \in \mbb{N}}$ of $A$-linear sheaf homomorphisms
$\psi_i : \mcal{M}_{i+1} \to \mcal{M}_i$.
The conditions are:
\begin{enumerate}
\rmitem{i} For every $i$ one has $\m^{i+1} \mcal{M}_i = 0$. Thus 
$\mcal{M}_i$ is a sheaf of $A_i$-modules.
\rmitem{ii} For every $i$ the $A_i$-linear sheaf homomorphism
$A_i \otimes_{A_{i+1}} \mcal{M}_{i+1} \to \mcal{M}_i$
induced by $\psi_i$ is an isomorphism. 
\end{enumerate}
\end{dfn}

\begin{rem}
When $A = \what{\mbb{Z}}_l$, the $l$-adic completion of $\mbb{Z}$
for some prime number $l$, this is called an $l$-adic sheaf.
Cf.\ \cite[Section 12]{FK}.
\end{rem}

\begin{cor} \label{cor:5}
Let $A$ be a complete noetherian local ring, with maximal
ideal $\m$, and let $\{ \mcal{M}_i \}_{i \in \mbb{N}}$ be an $\m$-adic system
of sheaves of $A$-modules on $X$. Assume these conditions hold\tup{:}
\begin{enumerate}
\rmitem{i} For every $i \geq 0$ the sheaf of $A_i$-modules
$\mcal{M}_i$ is flat.
\rmitem{ii} $X$ is locally $\mcal{M}_0$-simply connected.
\end{enumerate}
Define the sheaf of $A$-modules
$\mcal{M} := \opn{lim}_{\leftarrow i}\, \mcal{M}_i$.
Then the following are true.
\begin{enumerate}
\item $\mcal{M}$ is flat and $\m$-adically complete. 
\item For every $i$ the canonical sheaf homomorphism
$A_i \otimes_{A} \mcal{M} \to \mcal{M}_i$ is an isomorphism.
\item Let $U$ be an open set of $X$ that is 
$\mcal{M}_0$-simply connected. Then the $A$-module
$\Gamma(U, \mcal{M})$ is $\m$-adically free.
\end{enumerate}
\end{cor}

\begin{proof}
Let $U$ be an $\mcal{M}_0$-simply connected open set in $X$.
Write $M := \Gamma(U, \mcal{M})$ and 
$M_i := \Gamma(U, \mcal{M}_i)$, so that 
$M \cong \opn{lim}_{\leftarrow i}\, M_i$.
Fix some $j \geq 0$. Theorem \ref{thm:3}, applied to the artinian ring 
$A_j$ and the sheaf of $A_j$-modules $\mcal{M}_j|_U$, says
that $M_j$ is a free $A_j$-module, and for every
$i \leq j$ the canonical homomorphism
$A_i \otimes_{A_j} M_j \to M_i$ is bijective.
Hence the collection $\{ M_i \}_{i \geq 0}$ satisfies the assumptions
of Theorem \ref{thm:1}, and we conclude that $M$ is $\m$-adically
free over $A$, and $A_i \otimes_A M \cong M_i$ for every $i$.

We have shown that the canonical homomorphism
\[ A_i \otimes_A \Gamma(U, \mcal{M}) \to \Gamma(U, \mcal{M}_i) \]
is bijective for every open set $U$ that is $\mcal{M}_0$-simply connected.
Since these open sets form a basis of the topology,
it follows that 
$A_i \otimes_A \mcal{M} \to \mcal{M}_i$
is an isomorphism of sheaves for all $i$. 
So $\mcal{M}$ is $\m$-adically complete.

Finally we must prove that for any point $x \in X$ the stalk 
$\mcal{M}_x$ is a flat $A$-module. But
$\mcal{M}_x \cong 
\opn{lim}_{U \rightarrow} \, \Gamma(U, \mcal{M})$, 
where the limit is over the open neighborhoods of $x$ that are
$\mcal{M}_0$-simply connected. Since each 
$\Gamma(U, \mcal{M})$ is flat over $A$ (by Corollary \ref{cor:1}), so is their
direct limit.
\end{proof}

\begin{cor}
Suppose $B$ is another complete noetherian local ring,
with maximal ideal $\mfrak{n}$, and $A \to B$ is a local 
homomorphism. Let $\mcal{M}$ be a flat $\m$-adically complete
sheaf of $A$-modules on
$X$. Assume that $X$ is locally $\mcal{M}_0$-simply connected,
where $\mcal{M}_0 := A_0 \otimes_A \mcal{M}$. Then the sheaf of
$B$-modules $B \hatotimes{A} \mcal{M}$ is flat and 
$\mfrak{n}$-adically complete.
\end{cor}

\begin{proof}
Write
$\mcal{M}_i := A_i \otimes_A \mcal{M}$; so 
$\{ \mcal{M}_i \}_{i \in \mbb{N}}$ is an $\m$-adic system of
sheaves $A$-modules, and $\mcal{M}_i$ is flat over $A_i$.
Let 
$\mcal{N} := B \hatotimes{A} \mcal{M}$, 
$B_i := B / \mfrak{n}^{i+1}$ and
$\mcal{N}_i := B_i \otimes_{A_i} \mcal{M}_i$.
Then $\{ \mcal{N}_i \}_{i \in \mbb{N}}$ is an $\mfrak{n}$-adic system
of sheaves $B$-modules, $\mcal{N}_i$ is flat over $B_i$,
and $\mcal{N} \cong \opn{lim}_{\leftarrow i}\, \mcal{N}_i$.
Since $B_0$ is a free module over the field $A_0$, we have
\[ \mrm{H}^1(U, \mcal{N}_0) \cong
B_0 \otimes_{A_0} \mrm{H}^1(U, \mcal{M}_0) \]
for every open set $U$ in $X$. Therefore $X$ is locally $\mcal{N}_0$-simply
connected. Now we can use Corollary 
\ref{cor:5}.
\end{proof}


\end{document}